\documentclass[11pt]{amsart}

\allowdisplaybreaks

\theoremstyle{plain}
\newtheorem{theorem}{Theorem}[section]
\newtheorem{lemma}{Lemma}[section]

\theoremstyle{definition}
\newtheorem{conjecture}{Conjecture}[section]

\theoremstyle{remark}

\newtheorem{remark}{Remark}[section]

\newcommand{\cF}{{\mathcal{F}}}

        \newcommand{\field}[1]{{\mathbb{#1}}}
        \newcommand{\NN}{\field{N}}

        \newcommand{\RR}{\field{R}}

\newcommand{\Dom}{\mbox{\rm Dom}}

\newcommand{\Tr}{\mbox{\rm Tr}}

\begin{document}
\title[PERIODIC SCHR\"ODINGER OPERATORS WITH HYPERSURFACE WELLS]{SPECTRAL GAPS FOR
PERIODIC SCHR\"ODINGER OPERATORS WITH HYPERSURFACE MAGNETIC WELLS}

\author{B. HELFFER }
\address{D\'epartement de Math\'ematiques, B\^atiment 425, Univ
Paris-Sud et CNRS,\\ F-91405 Orsay C\'edex, France}
\email{Bernard.Helffer@math.u-psud.fr}

\author{Y. A. KORDYUKOV }
\address{Institute of Mathematics, Russian Academy of Sciences, 112 Chernyshevsky
str.\\ 450077 Ufa, Russia} \email{yurikor@matem.anrb.ru}
\thanks{B.H. was partially supported by the ESF programme SPECT.
Y.K. was partially supported by the Russian Foundation of Basic
Research (grant 06-01-00208).}

\begin{abstract}
We consider a periodic magnetic Schr\"odinger operator on a
noncompact Riemannian manifold $M$ such that $H^1(M, \RR)=0$ endowed
with a properly discontinuous cocompact isometric action of a
discrete group. We assume that there is no electric field and that the
magnetic field has a periodic set of compact magnetic wells. We
review a general scheme of a proof of existence of an arbitrary
large number of gaps in the spectrum of such an operator in the
semiclassical limit, which was suggested in our previous paper, and
some applications of this scheme. Then we apply these methods to
establish similar results in the case when the wells have regular
hypersurface pieces.
\end{abstract}


\maketitle

\section*{Introduction}
Let $ M$ be a noncompact oriented manifold of dimension $n\geq 2$
equipped with a properly discontinuous action of a finitely
generated, discrete group $\Gamma$ such that $M/\Gamma$ is
compact. Suppose that $H^1(M, \RR) = 0$, i.e. any closed $1$-form
on $M$ is exact. Let $g$ be a $\Gamma$-invariant Riemannian metric
and $\bf B$ a real-valued $\Gamma$-invariant closed 2-form on $M$.
Assume that $\bf B$ is exact and choose a real-valued 1-form $\bf
A$ on $M$ such that $d{\bf A} = \bf B$.

Thus, one has a natural mapping
\[
u\mapsto ih\,du+{\bf A}u
\]
from $C^\infty_c(M)$ to the space $\Omega^1_c(M)$ of smooth,
compactly supported one-forms on $M$. The Riemannian metric allows
to define scalar products in these spaces and consider the adjoint
operator
\[
(ih\,d+{\bf A})^* : \Omega^1_c(M)\to C^\infty_c(M).
\]
A Schr\"odinger operator with magnetic potential $\bf A$ is
defined by the formula
\[
H^h = (ih\,d+{\bf A})^* (ih\,d+{\bf A}).
\]
Here $h>0$ is a semiclassical parameter, which is assumed to be
small.

Choose local coordinates $X=(X_1,\ldots,X_n)$ on $M$. Write the
1-form $\bf A$ in the local coordinates as
\[
{\bf A}= \sum_{j=1}^nA_j(X)\,dX_j,
\]
the matrix of the Riemannian metric $g$ as
\[
g(X)=(g_{j\ell}(X))_{1\leq j,\ell\leq n}
\]
and its inverse as
\[
g(X)^{-1}=(g^{j\ell}(X))_{1\leq j,\ell\leq n}.
\]
Denote $|g(X)|=\det(g(X))$. Then the magnetic field $\bf B$ is
given by the following formula
\[
{\bf B}=\sum_{j<k}B_{jk}\,dX_j\wedge dX_k, \quad
B_{jk}=\frac{\partial A_k}{\partial X_j}-\frac{\partial
A_j}{\partial X_k}.
\]
Moreover, the operator $H^h$ has the form
\begin{multline*}
H^h=\frac{1}{\sqrt{|g(X)|}}\sum_{1\leq j,\ell\leq n}\left(i h
\frac{\partial}{\partial X_j}+A_j(X)\right)\\
\times \left[\sqrt{|g(X)|} g^{j\ell}(X) \left(i h
\frac{\partial}{\partial X_\ell}+A_\ell(X)\right)\right].
\end{multline*}

For any $x\in M$, denote by $B(x)$ the anti-symmetric linear
operator on the tangent space $T_x{ M}$ associated with the 2-form
$\bf B$:
\[
g_x(B(x)u,v)={\bf B}_x(u,v),\quad u,v\in T_x{ M}.
\]
Recall that the intensity of the magnetic field is defined as
\[
{\Tr}^+ (B(x))=\sum_{\substack{\lambda_j(x)>0\\ i\lambda_j(x)\in
\sigma(B(x)) }}\lambda_j(x)=\frac{1}{2}\Tr([B^*(x)\cdot
B(x)]^{1/2}).
\]
It turns out that in many problems the function $x\mapsto h\cdot
{\Tr}^+ (B(x))$ can be considered as a magnetic potential, that is,
as a magnetic analogue of the electric potential $V$ in a
Schr\"odinger operator $-h^2\Delta+V$.

We will also use the trace norm of $B(x)$:
\[
|B(x)|=[\Tr (B^*(x)\cdot B(x))]^{1/2}.
\]
It coincides with the norm of $B(x)$ with respect to the Riemannian
metric on the space of linear operators on $T_xM$ induced by the
Riemannian metric $g$ on $M$.

In this paper we will always assume that the magnetic field has a
periodic set of compact potential wells. More precisely, put
\[
b_0=\min \{{\Tr}^+ (B(x))\, :\, x\in { M}\}
\]
and assume that there exist a (connected) fundamental domain $\cF$
and a constant $\epsilon_0>0$ such that
\begin{equation}\label{YK:tr1}
  {\Tr}^+ (B(x)) \geq b_0+\epsilon_0, \quad x\in \partial\cF.
\end{equation}
For any $\epsilon_1 \leq \epsilon_0$, put
\[
U_{\epsilon_1} = \{x\in \cF\,:\, {\Tr}^+ (B(x)) < b_0+
\epsilon_1\}.
\]
Thus $U_{\epsilon_1}$ is an open subset of $\cF$ such that
$U_{\epsilon_1}\cap \partial\cF=\emptyset$ and, for $\epsilon_1 <
\epsilon_0$, $\overline{U_{\epsilon_1}}$ is compact and included in
the interior of $\cF$. Any connected component of $U_{\epsilon_1}$
with $\epsilon_1 < \epsilon_0$ and also any its translation under
the action of an element of $\Gamma$ can be understood as a magnetic
well. These magnetic wells are separated by potential barriers,
which are getting higher and higher when $h\to 0$ (in the
semiclassical limit).

For any linear operator $T$ in a Hilbert space, we will denote by
$\sigma(T)$ its spectrum. By a gap in the spectrum of a
self-adjoint operator $T$ we will mean any connected component of
the complement of $\sigma(T)$ in $\RR$, that is, any maximal
interval $(a,b)$ such that
\[
(a,b)\cap \sigma(T) = \emptyset\,.
\]

The problem of existence of gaps in the spectra of second order
periodic differential operators has been extensively studied
recently. Some related results on spectral gaps for periodic
magnetic Schr\"odinger operators can be found for example
in~\cite{BDP,gaps,diff2006,HS88,HSLNP345,HempelHerbst95,HempelPost02,HerbstNakamura,Ko04,bedlewo,KMS,MS,Nakamura95}
(see also the references therein).

In this paper, we consider the magnetic Schr\"odinger operator $H^h$
as an unbounded self-adjoint operator in the Hilbert space $L^2(M)$
and will study gaps in the spectrum of this operator, which are
located below the top of potential barriers, that is, on the
interval $[0, h(b_0+\epsilon_0)]$. In this case, the important role
is played by the tunneling effect, that is, by the possibility for
the quantum particle described by the Hamiltonian $H^h$ with such an
energy to pass through a potential barrier. Using the semiclassical
analysis of the tunneling effect, we showed in \cite{gaps} that the
spectrum of the magnetic Schr\"odinger operator $H^h$ on the
interval is localized in an exponentially small neighborhood of the
spectrum of its Dirichlet realization inside the wells. This result
reduces the investigation of gaps in the spectrum of the operator
$H^h$ to the study of the eigenvalue distribution for a ``one-well''
operator and leads us to suggest a general scheme of a proof of
existence of spectral gaps in \cite{diff2006}. We review this scheme
and some of its applications in Section~\ref{YK:quasi}. Then, in
Section~\ref{YK:hyper}, we will apply these methods to prove the
existence of an arbitrary large number of gaps in the spectrum of
the operator $H^h$,  as $h\rightarrow 0$, under the assumption that
$b_0=0$ and the zero set of $\mathbf B$ has regular codimension one
pieces.

\section{Quasimodes and spectral gaps}\label{YK:quasi}
In this section, we review a general scheme of a proof of existence
of gaps in the spectrum of the magnetic Schr\"odinger operator $H^h$
on the interval $[0, h(b_0+\epsilon_0)]$ and some of its
applications obtained in~\cite{diff2006}.

\subsection{A general scheme}
For any domain $W$ in $ M$, denote by $H^h_W$ the unbounded
self-adjoint operator in the Hilbert space $L^2(W)$ defined by the
operator $H^h$ in $\overline{W}$ with the Dirichlet boundary
conditions. The operator $H^h_W$ is generated by the quadratic form
\[
u\mapsto q^h_W [u] : = \int_W |(ih\,d+{\bf A})u|^2\,dx
\]
with the domain
\[
\Dom (q^h_W) = \{ u\in L^2(W) : (ih\,d+{\bf A})u \in L^2\Omega^1(W),
u\left|_{\partial W}\right.=0 \},
\]
where $L^2\Omega^1(W)$ denotes the Hilbert space of $L^2$
differential $1$-forms on $W$, $dx$ is the Riemannian volume form
on $M$.

Assume now that the operator $H^h$ satisfies the condition
\eqref{YK:tr1}. Fix $\epsilon_1>0$ and $\epsilon_2>0$ such that
$\epsilon_1 < \epsilon_2 < \epsilon_0$, and consider the operator
$H^h_D$ associated with the domain $D=\overline{U_{\epsilon_2}}$.
The operator $H^h_D$ has discrete spectrum.

The following result is a slight generalization of Theorem 2.1 in
\cite{diff2006}, which is concerned with the case when $N_h$ is
independent of $h$. It permits to get a more precise information on
the number of gaps as $h\rightarrow 0$.

\begin{theorem}\label{YK:abstract}
Suppose that there exist $h_0>0$, $c>0$, $M\geq 1$ and that, for
$h\in (0,h_0]$,  there exists $N_h$ and a subset
$\mu_0^h<\mu_1^h<\ldots <\mu_{N_h}^h$ of an interval $I(h)\subset
[0, h(b_0+\epsilon_1))$ such that
\begin{gather*}
\mu_{j}^h-\mu_{j-1}^h>ch^M, \quad j=1,\ldots,N_h,\\
{\rm dist}(\mu_0^h,\partial I(h))>ch^M,\quad {\rm
dist}(\mu_{N_h}^h,\partial I(h))>ch^M,
\end{gather*}
and, for each $j=0,1,\ldots,N_h$, there exists some non
  trivial  $v_j^h\in C^\infty_c(D)$ such that
\[
\|H^h_Dv_j^h-\mu^h_jv_j^h\|\leq \frac c3 \, h^M\, \|v_j^h\|\,.
\]
Then there exists  $h_1\in (0,h_0]$ such that the spectrum of $H^h$ on
the interval $I(h)$ has at least $N_h$ gaps for  $h\in (0,h_1)$.
\end{theorem}

\subsection{A generic situation}
As a first application of Theorem~\ref{YK:abstract}, we show in
\cite{diff2006} that the spectrum of the Schr\"o\-din\-ger operator
$H^h$, satisfying the assumption \eqref{YK:tr1}, always has gaps
(moreover, an arbitrarily large number of gaps) on the interval $[0,
h(b_0+\epsilon_0)]$ in the semiclassical limit $h\to 0$. Under some
additional generic assumption, this result was obtained in
\cite{gaps}. Indeed, slightly modifying the arguments of
\cite{diff2006}, one can show the following theorem.

\begin{theorem}
Under the assumption \eqref{YK:tr1}, for any interval
$[\alpha,\beta]\subset [b_0, b_0+\epsilon_0]$ and for any natural
$N$, there exists $h_0>0$ such that, for any $h\in (0,h_0]$, the
spectrum of $H^h$ in the interval $[h\alpha, h\beta]$ has at least
$N$ gaps.
\end{theorem}

The proof of this theorem can be given by a straightforward
repetition of the proof of Theorem 3.1 in \cite{diff2006} with the
only difference --- one should choose $\mu_0<\mu_1<\ldots<\mu_N$ in
the interval $(\alpha,\beta)$ instead of $(b_0, b_0+\epsilon_0)$.

Indeed, using Theorem~\ref{YK:abstract} with $N_h$ dependent on $h$
and a continuous family of quasimodes constructed in the proof of
Proposition 2.3 in \cite{gaps}, we can get an estimate for the
number of gaps in the constant rank case. Denote by $[a]$ the
integer part of $a$ (the largest integer $n$ satisfying $n\leq a$).

\begin{theorem}
Under the assumption \eqref{YK:tr1}, suppose that the rank of
$\mathbf B$ is constant in an open set $U\subset M$. Then, for any
interval $[\alpha,\beta]\subset \mathrm{Tr}^+ B(U)$, there exists
$h_0>0$ and $C>0$ such that, for any $h\in (0,h_0]$, the spectrum of
$H^h$ in the interval $[h\alpha, h\beta]$ has at least $[Ch^{-1/3}]$
gaps.
\end{theorem}

\subsection{The case of discrete wells} A more precise information on
location and asymptotic behavior of gaps in the spectrum of the
magnetic Schr\"odinger operator $H^h$, satisfying the assumption
\eqref{YK:tr1}, can be obtained, if we impose additional hypotheses
on the bottoms of magnetic wells. In this section, we consider a
case when the bottom of the magnetic well contains zero-dimensional
components, that is, isolated points, and, moreover, the magnetic
field behaves regularly near these points. More precisely, we will
assume that $b_0=0$ and that there is at least one zero $x_0$ of $B$
such that, for some integer $k>0$, there exists a positive constant
$C$ such that for all $x$ in some neighborhood of $x_0$ the estimate
holds:
\begin{equation}\label{YK:B}
C^{-1}d(x,x_0)^k\leq |B(x)|  \leq C d(x, x_0)^k
\end{equation}
(here $d(x,y)$ denotes the geodesic distance between $x$ and $y$).
In this case, the important role is played by a differential
operator $K^h_{\bar{x}_0}$ in $\RR^n$, which is in some sense an
approximation to the operator $H^h$ near $x_0$. Recall its
definition (see \cite{HM}).

Let $\bar x_0$ be a zero of $B$. Choose local coordinates $f: U(\bar
x_0)\to \RR^n$ on $M$, defined in a sufficiently small neighborhood
$U(\bar x_0)$ of $\bar x_0$. Suppose that $f(\bar x_0)=0$, and the
image $f(U(\bar x_0))$ is a ball $B(0,r)$ in $\RR^n$ centered at the
origin.

Write the 2-form $\bf B$ in the local coordinates as
\[
{\bf B}(X)=\sum_{1\leq \ell<m\leq n} b_{\ell m}(X)\,dX_\ell\wedge
dX_m, \quad X=(X_1,\ldots,X_n)\in B(0,r).
\]
Let ${\bf B}^0$ be the closed 2-form in $\RR^n$ with polynomial
components defined by the formula
\[
{\bf B}^0(X)=\sum_{1\leq \ell<m\leq
n}\sum_{|\alpha|=k}\frac{X^\alpha}{\alpha !}\frac{\partial^\alpha
b_{\ell m}}{\partial X^\alpha}(0)\,dX_\ell\wedge dX_m, \quad
X\in\RR^n.
\]
One can find a 1-form ${\bf A}^{0}$ on $\RR^n$ with polynomial
components such that
\[
d{\bf A}^0(X) ={\bf B}^0(X), \quad X\in\RR^n.
\]

Let $K^h_{\bar{x}_0}$ be a self-adjoint differential operator in
$L^2(\RR^n)$ with polynomial coefficients given by the formula
\[
K_{\bar{x}_0}^h =  (i h\,d+{\bf A}^0)^* (i h\,d+{\bf A}^0),
\]
where the adjoints are taken with respect to the Hilbert structure
in $L^2(\RR^n)$ given by the flat Riemannian metric $(g_{\ell
m}(0))$ in $\RR^n$. If ${\bf A}^0$ is written as
\[
{\bf A}^0=A^0_{1}\, dX_1+\ldots+ A^0_{n}\,dX_n,
\]
then $K^h_{\bar{x}_0}$ is given by the formula
\[
K_{\bar{x}_0}^h=\sum_{1\leq \ell,m\leq n} g^{\ell m}(0) \left(i h
\frac{\partial}{\partial X_\ell}+A^0_{\ell}(X)\right)\left(i h
\frac{\partial}{\partial X_m}+A^0_{m}(X)\right).
\]
The operators $K^h_{\bar{x}_0}$ have discrete spectrum (cf, for
instance, \cite{HelNo85,HM88}). Using the simple dilation $X\mapsto
h^{\frac{1}{k+2}}X$, one can show that the operator
$K^h_{\bar{x}_0}$ is unitarily equivalent to
$h^{\frac{2k+2}{k+2}}K^1_{\bar{x}_0}$. Thus,
$h^{-\frac{2k+2}{k+2}}K^h_{\bar{x}_0}$ has discrete spectrum,
independent of $h$.

\begin{theorem}[\cite{diff2006}]\label{YK:gaps0}
Suppose that the operator $H^h$ satisfies the condition
\eqref{YK:tr1} with some $\epsilon_0>0$ and that there exists a zero
$\bar{x}_0$ of $B$, satisfying the assumption \eqref{YK:B} for some
integer $k>0$. Denote by $\lambda_1<\lambda_2<\lambda_3<\ldots$ the
eigenvalues of the operator $K^1_{\bar{x}_0}$ (not taking into
account multiplicities). Then, for any natural $N$ and any
$C>\lambda_{N+1}$,  there exists $h_0>0$ such that the spectrum of
$H^h$ in the interval $[0,Ch^{\frac{2k+2}{k+2}}]$ has at least $N$
gaps for any $h\in (0,h_0)$.
\end{theorem}

\section{Hypersurface wells}\label{YK:hyper}
In this section, we consider the case when $b_0=0$ and the zero set
of the magnetic field has regular codimension one pieces. More
precisely, suppose that there exists $x_0\in M$ such that $\mathbf
B(x_0)=0$ and in a neighborhood $U$ of $x_0$ the zero set of
$\mathbf B$ is a smooth oriented hypersurface $S$, and, moreover,
there are constants $k\in \NN$ and $C>0$ such that for all $x\in U$
we have:
\begin{equation}\label{YK:B1}
C^{-1}d(x,S)^k\leq |B(x)|  \leq C d(x,S)^k\,.
\end{equation}

On compact manifolds, this model was introduced for the first time
by Montgomery~\cite{Mont} and was further studied in
\cite{HM,Pan-Kwek,syrievienne}.

Let
\[
\omega_{0.0}=i^*_S{\mathbf A}
\]
be the closed one form on $S$ induced by ${\mathbf A}$, where $i_S$
is the embedding of $S$ to $M$.

Denote by $N$ the external unit normal vector to $S$ and by
$\tilde{N}$ an arbitrary extension of $N$ to a smooth vector field
on $U$.

Let $\omega_{0,1}$ be the smooth one form on $S$ defined, for any
vector field $V$ on $S$, by the formula
\[
\langle V,\omega_{0,1}\rangle(y)=\frac{1}{k!}\tilde{N}^k({\mathbf
B}(\tilde{N},\tilde{V}))(y), \quad y\in S,
\]
where $\tilde{V}$ is a $C^\infty$ extension of $V$ to $U$. By
\eqref{YK:B1}, it is easy to see that $\omega_{0,1}(x)\not=0$ for
any $x\in S$. Denote
\[
\omega_{\mathrm{min}}(B)=\inf_{x\in S} |\omega_{0,1}(x)|>0.
\]

For any $\alpha\in \RR$ and $\beta\in\RR, \beta\neq 0$, consider the
self-adjoint second order differential operator in $L^2(\RR)$ given
by
\[
Q(\alpha, \beta)=-\frac{d^2 }{d t^2}+ \left(\frac{1}{k+1}\beta
t^{k+1}- \alpha \right)^2.
\]
In the context of magnetic bottles, this family of operators (for
$k=1$) first appears in \cite{Mont} (see also \cite{HM}). Denote by
$\lambda_0(\alpha,\beta)$ the bottom of the spectrum of the operator
$Q(\alpha,\beta)$.

Let us recall some properties of $\lambda_0(\alpha,\beta)$, which
were established in \cite{Mont,HM,Pan-Kwek}. First of all, remark
that $\lambda_0(\alpha,\beta)$ is a continuous function of
$\alpha\in \RR$ and $\beta\in \RR\setminus\{0\}$. One can see by
scaling that, for $\beta>0$,
\begin{equation}\label{YK:alphabeta}
\lambda_0(\alpha,\beta)= \beta ^{\frac{2}{k+2}} \lambda_0
(\beta^{-\frac{1}{k+2}}\alpha,1)\;.
\end{equation}
A further discussion depends on $k$ odd or $k$ even.

When $k$ is odd, $\lambda_0(\alpha,1)$ tends to $+\infty$ as $\alpha
\rightarrow -\infty$ by monotonicity. For analyzing its behavior as
$\alpha\rightarrow +\infty$, it is suitable to do a dilation
$t=\alpha^{\frac{1}{k+1}} s$, which leads to the analysis of
\[
\alpha^2 \left(-h^2\frac{d^2}{ds^2} + \left(\frac{s^{k+1}}{k+1}
-1\right)^2 \right)
\]
with $h=\alpha^{-(k+2)/(k+1)}$ small. Semi-classical analysis is
relevant, and it is easy to show, using harmonic approximation, that
\[
\lambda_0(\alpha,1) \sim (k+1)^{\frac{2k}{k+1}}
\alpha^{\frac{k}{k+1}}\;,\;\mbox{ as } \alpha \rightarrow +\infty\;.
\]
In particular, we see that $\lambda_0(\alpha,1)$ tends to $+\infty$.

When $k$ is even, we have $
\lambda_0(\alpha,1)=\lambda_0(-\alpha,1)$, and, therefore, it is
sufficient to consider the case $\alpha \geq 0$. As $\alpha
\rightarrow +\infty$, semi-classical analysis again shows that
$\lambda_0(\alpha,1)$ tends to $+\infty$.

So in both cases, it is clear that the continuous function
$\lambda_0(\alpha,1)$ is lower semi-bounded:
\[
\hat{\nu}:=\inf_{\alpha\in \RR}\lambda_0(\alpha,1)>-\infty,
\]
and there exists (at least one) $\alpha_{\mathrm{min}}\in \RR$ such
that $\lambda_0(\alpha,1)$ is minimal:
\[
\lambda_0(\alpha_{\mathrm{min}},1)=\hat{\nu}.
\]
For $k$ odd, one can show that the minimum $\alpha_{\mathrm{min}}$
is strictly positive. One can indeed compute the derivative of
$\lambda_0(\alpha,1)$ at $\alpha =0$ and find that
$$
\frac{\partial \lambda_0}{\partial\alpha}(0,1) < 0\;.
$$
In the case $k=1$, it has been shown that this minimum is unique
(see \cite{Pan-Kwek}). Numerical computations show (see
\cite{Mont,HM}) that, in this case, $\hat\nu\cong 0.5698$.

\begin{theorem}\label{YK:hypersurface}
For any $a$ and $b$ such that
\[
\hat{\nu}\, \omega_{\mathrm{min}}(B)^{\frac{2}{k+2}}<a < b
\]
and for any natural $N$, there exists $h_0>0$ such that, for any
$h\in (0,h_0]$, the spectrum of $H^h$ in the interval
\[
[h^{\frac{2k+2}{k+2}}a, h^{\frac{2k+2}{k+2}}b]
\]
has at least $N$ gaps.
\end{theorem}

\begin{proof}
Let $g_0$ be the Riemannian metric on $S$ induced by $g$. Without
loss of generality, we can assume that $U$ coincides with an open
tubular neighborhood of $S$ and choose a diffeomorphism
\[
\Theta : I\times S\to U,
\]
where $I$ is an open interval $(-\varepsilon_0,\varepsilon_0)$ with
$\varepsilon_0>0$ small enough, such that $\Theta\left|_{\{0\}\times
S}\right.=\mathrm{id}$ and
\[
(\Theta^*g-\tilde{g}_0)\left|_{\{0\}\times S}\right.=0,
\]
where $\tilde{g}_0$ is a Riemannian metric on $I\times S$ given by
\[
\tilde{g}_0=dt^2+g_0.
\]
By adding to $\mathbf A$ the exact one form $d\phi$, where $\phi$ is
the function satisfying
\begin{gather*}
    N(x)\phi(x)=-\langle N,{\bf A}\rangle(x), \quad x\in U,\\
    \phi(x) =0, \quad x\in S,
\end{gather*}
we may assume that
\[
\langle N,{\bf A}\rangle(x)=0, \quad x\in U.
\]

As above, denote by $H_D^h$ the unbounded self-adjoint operator in
$L^2(D)$ given by the operator $H^h$ in the domain $D=\overline{U}$
with the Dirichlet boundary conditions.

For any $t\in \RR$, let
$P^h_S\left(\omega_{0,0}+\frac{1}{k+1}t^{k+1}\omega_{0,1}\right)$ be
a formally self-adjoint operator in $L^2(S, dx_{g_0})$ defined by
\begin{multline*}
P^h_S\left(\omega_{0,0}+\frac{1}{k+1}t^{k+1}\omega_{0,1}\right)=
\left(ihd+\omega_{0,0}+\frac{1}{k+1}t^{k+1}\omega_{0,1}\right)^*\\
\times\left(ihd+\omega_{0,0}+\frac{1}{k+1}t^{k+1}\omega_{0,1}\right).
\end{multline*}

Consider the self-adjoint operator $H^{h,0}$ in $L^2(\RR\times S,
dt\,dx_{g_0})$ defined by the formula
\[
H^{h,0}=-h^2\frac{\partial^2 }{\partial t^2}+
P^h_S\left(\omega_{0,0}+\frac{1}{k+1}t^{k+1}\omega_{0,1}\right)
\]
with the Dirichlet boundary conditions. By Theorem 2.7 of \cite{HM},
the operator $H^{h,0}$ has discrete spectrum. Moreover, it can be
seen from the proof of this theorem that if $\lambda^0(h)$ is an
approximate eigenvalue of $H^{h,0}$ with the corresponding
approximate eigenfunction $w^h\in C^\infty_c(\RR\times S)$ such that
\[
\lambda^0(h)\leq D h^{(2k+2)/(k+2)}
\]
and
\[
\|(H^{h,0}-\lambda^0(h))w^h\|\leq Ch^{(2k+3)/(k+2)}\|w^h\|,
\]
then $\lambda^0(h)$ is an approximate eigenvalue of $H_D^h$ with the
corresponding approximate eigenfunction $v^h=(\Theta^{-1})^*w^h\in
C^\infty_c(U)$:
\[
\|(H_D^h-\lambda^0(h))v^h\|\leq Ch^{(2k+3)/(k+2)}\|v^h\|.
\]
So it remains to construct approximate eigenvalues of $H^{h,0}$.

\begin{lemma}\label{YK:l1}
For any $\lambda \geq\hat{\nu} \,\omega_{\mathrm{min}}(B)^{2/(k+2)}$,
there exists $\Phi\in C^\infty_c(\RR\times S)$ such that
\[
\|(H^{h,0} - \lambda h^{\frac{2k+2}{k+2}})\Phi\|\leq C
h^{\frac{6k+8}{3(k+2)}}\|\Phi\|.
\]
\end{lemma}

\begin{proof}
Take $x_1\in S$ such that $|\omega_{0,1}(x_1)|=
\omega_{\mathrm{min}}(B)$. Consider $\alpha_1\in \RR$ such that
$\lambda_0(\alpha_1,1)=\lambda\omega_{\mathrm{min}}(B)^{-2/(k+2)}
\geq \hat{\nu}$. Let $\psi\in L^2(\RR)$ be a normalized
eigenfunction of $Q(\alpha_1,1)$, corresponding to
$\lambda_0(\alpha_1,1)$:
\[
\left[-\frac{d^2 }{d t^2}+ \left(\frac{1}{k+1}
t^{k+1}-\alpha_1\right)^2\right]\psi(t)=\lambda
\omega_{\mathrm{min}}(B)^{-\frac{2}{k+2}}\psi(t), \quad
\|\psi\|_{L^2(\RR)}=1.
\]
Then the function
\[
\Psi(t)=\omega_{\mathrm{min}}(B)^{\frac{1}{2(k+2)}}
h^{-\frac{1}{2(k+2)}}
\psi(\omega_{\mathrm{min}}(B)^{\frac{1}{k+2}}h^{-\frac{1}{k+2}}t)
\]
satisfies
\begin{multline*}
\left(-h^2\frac{d^2 }{d t^2}+
\left(\frac{1}{k+1}\omega_{\mathrm{min}}(B)
t^{k+1}-\alpha_1\omega_{\mathrm{min}}(B)^{\frac{1}{k+2}}
h^{\frac{k+1}{k+2}} \right)^2\right)\Psi(t)
\\ =\lambda h^{\frac{2k+2}{k+2}}\Psi(t), \quad \|\Psi\|_{L^2(\RR)}=1.
\end{multline*}
Take normal coordinates $f: U(x_1) \subset S \to \RR^{n-1}$ on $S$
defined in a neighborhood $U(x_1)$ of $x_1$, where $f(U(x_1))
=B(0,r)$ is a ball in $\RR^{n-1}$ centered at the origin and
$f(x_1)=0$. Choose a function $\phi\in C^\infty(B(0,r))$ such that
$d\phi=\omega_{0,0}$. Write $\omega_{0,1}=
\sum_{j=1}^{n-1}\omega_j(s)\,ds_j.$ Note that
\[
\omega_{\mathrm{min}}(B)=\left(\sum\limits_{j=1}^{n-1}|\omega_j(0)|^2\right)^{1/2}.
\]
Consider the function $\Phi\in C^\infty(B(0,r)\times \RR)$ given by
\begin{multline}\label{YK:Phi}
\Phi(s,t)=c h^{-\beta/2(n-1)}\chi(s)\exp\left(-i\frac{\phi(s)}{h}
\right)\exp\left( i\frac{\alpha_1
\sum\limits_{j=1}^{n-1}\omega_j(0)s_j}{\omega_{\mathrm{min}}(B)^{\frac{k+1}{k+2}}
h^{\frac{1}{k+2}}}\right)\\ \times \exp
\left(-\frac{|s|^2}{2h^{2\beta}} \right) \Psi(t), \quad s\in
B(0,r),\quad t\in \RR,
\end{multline}
with some $\beta$, where $\chi\in C^\infty_c(B(0,r))$ is a cut-off
function, $c$ is chosen in such a way that $\|\Phi\|_{L^2(S\times
\RR)}=1$.

Put
\begin{align*}
H^{h,1}& =-h^2\frac{\partial^2 }{\partial t^2}+
P^h_S\left(\frac{1}{k+1}t^{k+1}\omega_{0,1}-\alpha_1
\omega_{\mathrm{min}}(B)^{-\frac{k+1}{k+2}}
h^{\frac{k+1}{k+2}}\omega_{0,1}(0) \right),\\
E(s)& =c h^{-\beta/2(n-1)}\chi(s) \exp
\left(-\frac{|s|^2}{2h^{2\beta}} \right), \\
\Phi_1(s,t)& =E(s)\Psi(t).
\end{align*}
Then we have
\[
H^{h,0}\Phi(s,t) = \exp\left(-i\frac{\phi(s)}{h} \right)\exp\left(
i\frac{\alpha_1
\sum\limits_{j=1}^{n-1}\omega_j(0)s_j}{\omega_{\mathrm{min}}(B)^{\frac{k+1}{k+2}}
h^{\frac{1}{k+2}}} \right) H^{h,1}\Phi_1(s,t).
\]
Next, we have
\begin{align*}
\lefteqn{P^h_S\left(\frac{1}{k+1}t^{k+1}\omega_{0,1}-\alpha_1\omega_{\mathrm{min}}(B)^{-\frac{k+1}{k+2}}
h^{\frac{k+1}{k+2}}\omega_{0,1}(0) \right)}\\
= & \sum_{j,\ell}
\frac{1}{\sqrt{g_0}}\left(ih\frac{\partial}{\partial
s_j}+\frac{1}{k+1}t^{k+1}\omega_{j}(s)-\alpha_{1}\omega_{\mathrm{min}}(B)^{-\frac{k+1}{k+2}}
h^{\frac{k+1}{k+2}}\omega_{j}(0)\right)\\ & \times \left(
g^{j\ell}_0\sqrt{g_0}\left(ih\frac{\partial}{\partial
s_\ell}+\frac{1}{k+1}t^{k+1}\omega_{\ell}(s)-\alpha_{1}\omega_{\mathrm{min}}(B)^{-\frac{k+1}{k+2}}
h^{\frac{k+1}{k+2}}\omega_{\ell}(0)\right) \right)\\
= & \sum_{j,\ell}g^{j\ell}_0 \left(ih\frac{\partial}{\partial
s_j}+\frac{1}{k+1}t^{k+1}\omega_{j}(s)-\alpha_{1}\omega_{\mathrm{min}}(B)^{-\frac{k+1}{k+2}}
h^{\frac{k+1}{k+2}}\omega_{j}(0)\right)\\
& \times\left(ih\frac{\partial}{\partial
s_\ell}+\frac{1}{k+1}t^{k+1}\omega_{\ell}(s)-\alpha_{1}\omega_{\mathrm{min}}(B)^{-\frac{k+1}{k+2}}
h^{\frac{k+1}{k+2}}\omega_{\ell}(0)\right)\\
& +\sum_{\ell}ih\Gamma^\ell(s) \left(ih\frac{\partial}{\partial
s_\ell}+\frac{1}{k+1}t^{k+1}\omega_{\ell}(s)-\alpha_{1}\omega_{\mathrm{min}}(B)^{-\frac{k+1}{k+2}}
h^{\frac{k+1}{k+2}}\omega_{\ell}(0)\right)\\
= & -h^2\sum_{j,\ell}g^{j\ell}_0 \frac{\partial^2}{\partial
s_j\partial s_\ell}+ 2ih \sum_{j,\ell}g^{j\ell}_0
\frac{1}{k+1}t^{k+1} \frac{\partial \omega_{\ell}}{\partial s_j}(s)\\
& + 2ih \sum_{j,\ell}g^{j\ell}_0
\left(\frac{1}{k+1}t^{k+1}\omega_{\ell}(s)-\alpha_{1}\omega_{\mathrm{min}}(B)^{-\frac{k+1}{k+2}}
h^{\frac{k+1}{k+2}}\omega_{\ell}(0)\right) \frac{\partial}{\partial s_j} \\
& + \sum_{j,\ell}g^{j\ell}_0
\left(\frac{1}{k+1}t^{k+1}\omega_{j}(s)-\alpha_{1}\omega_{\mathrm{min}}(B)^{-\frac{k+1}{k+2}}
h^{\frac{k+1}{k+2}}\omega_{j}(0)\right)\\
& \times \left(\frac{1}{k+1}t^{k+1}\omega_{\ell}(s)
-\alpha_{1}\omega_{\mathrm{min}}(B)^{-\frac{k+1}{k+2}}
h^{\frac{k+1}{k+2}}\omega_{\ell}(0)\right) -h^2
\sum_{\ell}\Gamma^\ell(s)\frac{\partial}{\partial s_\ell} \\
& +ih \sum_{\ell}\Gamma^\ell(s)\left(
\frac{1}{k+1}t^{k+1}\omega_{\ell}(s)-\alpha_{1}\omega_{\mathrm{min}}(B)^{-\frac{k+1}{k+2}}
h^{\frac{k+1}{k+2}}\omega_{\ell}(0)\right),
\end{align*}
where
\[
\Gamma^\ell =\sum_{j}\frac{1}{\sqrt{g_0}} \frac{\partial}{\partial
s_j}\left( g^{j\ell}_0\sqrt{g_0}\right).
\]
By a well-known property of normal coordinates, we have
$\partial_jg^{\ell m}_0(0)=0$. So we get $\Gamma^\ell(0)=0$, and
\begin{equation}\label{YK:1}
g^{\ell m}_0(s)=\delta^{\ell m}+O(|s|^2),\quad
\Gamma^\ell(s)=O(|s|), \quad s\to 0.
\end{equation}
We get
\begin{multline*}
H^{h,1}\Phi_1(s,t) = \lambda h^{\frac{2k+2}{k+2}}\Phi_1(s,t)
-h^2\sum_{j,\ell}g^{j\ell}_0 \frac{\partial^2 E}{\partial
s_j\partial s_\ell} (s)\Psi(t)\\
\begin{aligned}
& + 2ih \sum_{j,\ell}g^{j\ell}_0 \frac{1}{k+1} \frac{\partial
\omega_{\ell}}{\partial s_j}E(s)t^{k+1}\Psi(t) \\ & + 2ih
\sum_{j,\ell}g^{j\ell}_0 \frac{\partial E}{\partial s_j}(s)
\left(\frac{1}{k+1}t^{k+1}
\omega_{\ell}(s)-\alpha_{1}\omega_{\mathrm{min}}(B)^{-\frac{k+1}{k+2}}
h^{\frac{k+1}{k+2}}\omega_{\ell}(0)\right)\Psi(t) \\
& + R(s,t)E(s)\Psi(t)-h^2
\sum_{\ell}\Gamma^\ell(s)\frac{\partial}{\partial s_\ell}E(s)\Psi(t)
\\ & +ih \sum_{\ell}\Gamma^\ell(s)\left(
\frac{1}{k+1}t^{k+1}\omega_{\ell}(s)-\alpha_{1}\omega_{\mathrm{min}}(B)^{-\frac{k+1}{k+2}}
h^{\frac{k+1}{k+2}}\omega_{\ell}(0)\right)E(s)\Psi(t),
\end{aligned}
\end{multline*}
where
\begin{align*}
R(s,t)= & \sum_{j,\ell}g^{j\ell}_0
\left(\frac{1}{k+1}t^{k+1}\omega_{j}(s)-\alpha_{1}\omega_{\mathrm{min}}(B)^{-\frac{k+1}{k+2}}
h^{\frac{k+1}{k+2}}\omega_{j}(0)\right)\\ & \times
\left(\frac{1}{k+1}t^{k+1}\omega_{\ell}(s)
-\alpha_{1}\omega_{\mathrm{min}}(B)^{-\frac{k+1}{k+2}}
h^{\frac{k+1}{k+2}}\omega_{\ell}(0)\right)
\\ & - \sum_{j}
\left(\frac{1}{k+1}t^{k+1}\omega_{\mathrm{min}}(B)^2-\alpha_{1}\omega_{\mathrm{min}}(B)^{\frac{1}{k+2}}
h^{\frac{k+1}{k+2}}\right)^2 \\
= & \frac{1}{(k+1)^2}\left(\sum_{j,\ell}g^{\ell}_0(s)
\omega_{j}(s)\omega_{\ell}(s) - \sum_j (\omega_{j}(0))^2\right)
t^{2(k+1)}
\\ & - \frac{2}{k+1}t^{k+1} \sum_j(\omega_{j}(s)
-\omega_{j}(0))\alpha_{1}\omega_{\mathrm{min}}(B)^{-\frac{k+1}{k+2}}
h^{\frac{k+1}{k+2}}\omega_{j}(0) \\
& +  O(|s|^2) \sum_{j}
\left(\frac{1}{k+1}t^{k+1}\omega_{\mathrm{min}}(B)-\alpha_{1}\omega_{\mathrm{min}}(B)^{\frac{1}{k+2}}
h^{\frac{k+1}{k+2}}\right)^2.
\end{align*}
We have
\begin{equation}\label{YK:2}
\||s|^m E(s)\| = \left(h^{-\beta(n-1)} \int_{\RR^{n-1}} |s|^{2m}\exp
\left(-\frac{|s|^2}{h^{2\beta}} \right) ds \right)^{1/2}\\ =C_1
h^{\beta m},
\end{equation}
and, furthermore,
\begin{equation}\label{YK:3}
\||s|^m \frac{\partial E}{\partial s_j}(s)\| = C_2h^{\beta (m-1)},
\quad \||s|^m \frac{\partial^2 E}{\partial s_j \partial s_\ell}(s)\|
= C_3h^{\beta (m-2)}.
\end{equation}
We also have
\begin{equation}\label{YK:4}
\|t^{k+1}\Psi(t)\|\leq C_4h^{\frac{k+1}{k+2}}, \quad
\|t^{2(k+1)}\Psi(t)\|\leq C_5h^{\frac{2k+2}{k+2}}.
\end{equation}
Since $s=0$ is a minimum of $|\omega_{0,1}(s)|^2$, we have
\begin{equation}\label{YK:5}
|\omega_{0,1}(s)|^2-\omega_{\mathrm{min}}(B)^2=\sum_{j,\ell}g^{j\ell}_0(s)
\omega_{j}(s)\omega_{\ell}(s) -(\omega_{j}(0))^2 \leq C_6|s|^2
\end{equation}
and
\[
\left( \frac{\partial }{\partial s_r}|\omega_{0,1}|^2
\right)(0)=2\sum_{j} \frac{\partial \omega_{j}}{\partial
s_r}(0)\omega_{j}(0)=0,
\]
that implies
\begin{equation}\label{YK:6}
\left|\sum_j (\omega_{j}(s) -\omega_{j}(0))\omega_{j}(0)\right| \leq
C_7|s|^2.
\end{equation}
Using  \eqref{YK:1}, \eqref{YK:2}, \eqref{YK:3}, \eqref{YK:4},
\eqref{YK:5} and \eqref{YK:6} and putting $\beta=\frac{1}{3(k+2)}$,
one can easily get that
\[
\|H^{h,0}\Phi - \lambda h^{\frac{2k+2}{k+2}}\Phi\|\leq C
h^{\frac{6k+8}{3(k+2)}}.
\]
\end{proof}

Given $a$ and $b$ such that $\hat{\nu}\,
\omega_{\mathrm{min}}(B)^{2/(k+2)} <a < b$ and some natural $N$,
choose some finite sequence $\{\nu_j\}_{j=0,\dots,N}$ such that
\[
a<\nu_0<\nu_1<\ldots<\nu_N< b.
\]
By Lemma~\ref{YK:l1}, for any $m=0,1,\ldots,N$,
\[
\mu_m^h=\nu_m h^{\frac{2k+2}{k+2}}\in [h^{(2k+2)/(k+2)}a,
h^{(2k+2)/(k+2)}b]
\]
is an approximate eigenvalue of the operator $H^h_D$: for some
$\Phi^h_m\in C^\infty_c(D)$
\[
\|(H^{h,0} - \mu_m^h) \Phi^h_m\|\leq C h^{\frac{6k+8}{3(k+2)}}\|
\Phi^h_m \|.
\]
Using Theorem~\ref{YK:abstract} with $N_h=N$ independent of $h$, we
complete the proof.
\end{proof}

\begin{remark}
Using the methods of the proof of Theorem~\ref{YK:hypersurface}, one
can construct much more approximate eigenvalues of the operator
$H^h$ on the interval $[h^{(2k+2)/(k+2)}a, h^{(2k+2)/(k+2)}b]$ with
some $a$ and $b$ such that $\hat{\nu}\,
\omega_{\mathrm{min}}(B)^{2/(k+2)}<a < b$. Applying then
Theorem~\ref{YK:abstract} with $N_h$ dependent on $h$, one can get
the following theorem.

\begin{theorem}
Under the assumptions of Theorem~\ref{YK:hypersurface}, for any $a$
and $b$ such that
\[
\hat{\nu}\, \omega_{\mathrm{min}}(B)^{\frac{2}{k+2}}<a < b,
\]
there exist $h_0>0$ and $C>0$ such that, for any $h\in (0,h_0]$, the
spectrum of $H^h$ in the interval
\[
[h^{\frac{2k+2}{k+2}}a, h^{\frac{2k+2}{k+2}}b]
\]
has at least $[Ch^{-\frac{2}{3(k+2)}}]$ gaps.
\end{theorem}
\end{remark}

\section{Concluding remarks}
1. Suppose that the operator $H^h$ satisfies the condition
\eqref{YK:tr1} with some $\epsilon_0>0$, and the zero set of the
magnetic field $\mathbf B$ is a smooth oriented hypersurface $S$.
Moreover, assume that there are constants $k\in \NN$ and $C>0$ such
that for all $x$ in a neighborhood of $S$ we have:
\[
C^{-1}d(x,S)^k\leq |B(x)|  \leq C d(x,S)^k\,.
\]
It is interesting to determine the bottom $\lambda_0(H^{h})$ of the
spectrum of the operator $H^h$ in $L^2(M)$. By Theorem 2.1 in
\cite{gaps} and Theorem 2.7 in \cite{HM}, $\lambda_0(H^{h})$ is
asymptotically equal to the bottom $\lambda_0(H^{h,0})$ of the
spectrum of the operator $H^{h,0}$. From the construction of
approximate eigenvalues of $H^{h,0}$ given in Lemma~\ref{YK:l1}, one
can see that, in order to find $\lambda_0(H^{h,0})$, it is natural
to consider a self-adjoint second order differential operator
$P(\mathbf{v},\mathbf{w})$, $\mathbf{v},\mathbf{w}\in \RR^{n-1}$, in
$L^2(\RR)$ given by
\[
P(\mathbf{v},\mathbf{w})=-\frac{d^2 }{d t^2}+
\left|\frac{1}{k+1}\mathbf{w} t^{k+1}-\mathbf{v}\right|^2
\]
and minimize the bottoms $\lambda_0(\mathbf{v}, \mathbf{w})$ of the
spectrum of the operator $P(\mathbf{v},\mathbf{w})$ over
$\mathbf{v}\in \RR^{n-1}$ and $\mathbf{w}\in K$, where
$K=\{\omega_{0,1}(s) : s\in \bar{S}\}$ is a compact subset of
$\RR^{n-1}\setminus \{0\}$.

The identity
\[
P(\mathbf{v},\mathbf{w})=\left(-\frac{d^2 }{d t^2}+
\left(\frac{1}{k+1}|\mathbf{w}| t^{k+1}- \frac{\mathbf{v}\cdot
\mathbf{w}}{|\mathbf{w}|} \right)^2\right) + \left|\mathbf{v}-
\frac{\mathbf{v}\cdot \mathbf{w}}{|\mathbf{w}|^2}\mathbf{w}
\right|^2
\]
shows that, for determining the minimum of $\lambda_0(\mathbf{v},
\mathbf{w})$ over $\mathbf{v}\in \RR^{n-1}$ and $\mathbf{w}\in K$,
it is sufficient to assume that $\mathbf{v}$ is parallel to
$\mathbf{w}$. For such $\mathbf{v}$ and $\mathbf{w}$, we obtain
$P(\mathbf{v},\mathbf{w})=Q(\alpha, \beta)$ with $\alpha=\pm
|\mathbf{v}|,\beta=|\mathbf{w}|$. By \eqref{YK:alphabeta}, it
follows that, for determining the minimum of $\lambda_0(\alpha,
\beta)$ over $\alpha\in \RR$ and $\beta\in \{|\omega_{0,1}(s)| :
s\in \bar{S}\}$, we should first minimize over $\beta$, that is,
take $s_1\in S$ such that
\[
|\omega_{0,1}(s_1)|=\min \{|\omega_{0,1}(s)| : s\in \bar{S}\},
\]
and then, for the minimal $\beta$, minimize over $\alpha$.

This observation provides some explanations of our construction of
approximate eigenvalues of the operator $H^{h,0}$ in
Lemma~\ref{YK:l1}, in particular, of our choice for the exponent in
\eqref{YK:Phi}. It also motivates us to formulate the following
conjecture:

\begin{conjecture}
Under current assumptions, for the bottom $\lambda_0(H^h)$ of the
spectrum of the operator $H^h$ in $L^2(M)$, we have
\[
\lim_{h\to 0} h^{-\frac{2k+2}{k+2}}\lambda_0(H^h)=\hat{\nu}\,
\omega_{\mathrm{min}}(B)^{\frac{2}{k+2}}.
\]
\end{conjecture}

Observe that a similar result was obtained by Pan and
Kwek~\cite{Pan-Kwek} for the bottom of the spectrum of the Neumann
realization of the operator $H^h$ in a bounded domain in the case
$k=1$.

2. In the setting of Section~\ref{YK:hyper}, one can assume that the
function $|\omega_{0,1}(x)|$ has a non-degenerate minimum at some
$x_1\in S$. In some sense this is the ``miniwells case'' analyzed in
\cite{HelSj5} in comparison with the ``uniform case'' analyzed in
\cite{HelSj6}, which in this setting was studied in \cite{diff2006}.
Then we can obtain a more precise information about gaps located
near the bottom of the spectrum of $H^h$ (see some relevant
calculations in \cite{syrievienne}). This will be discussed
elsewhere.

3. The similar results both like in Section~\ref{YK:hyper} and like
in the ``miniwells case'' mentioned in the previous remark can be
obtained when $b_0\neq 0$ (for the ``miniwells case'', see some
relevant results in \cite{HM01}). We will consider these problems in
a future publication (see \cite{HKo}).

\end{document}